\documentclass[11pt,reqno]{amsart}
\usepackage{hyperref}
\usepackage{ amssymb }
\numberwithin{equation}{section}
\usepackage{float}
\usepackage{pgf,tikz}

\newcommand \reg{\operatorname{reg}}

\newcommand\link{\operatorname{link}}

\newcommand \ma{\operatorname{m}}

\newcommand \C{\mathcal{C}}
\newcommand \G{\mathcal{G}}

\newcommand \A{\mathcal{A}}
\newcommand \B{\mathcal{B}}

\newcommand \kk{\mathcal{K}}

\newtheorem{theorem}{Theorem}[section]
\newtheorem{definition}[theorem]{Definition}
\newtheorem{cons}[theorem]{Construction}
\newtheorem{lemma}[theorem]{Lemma}

\newtheorem{example}[theorem]{Example}
\newtheorem{obs}[theorem]{Observation}
\newtheorem{question}[theorem]{Question}
\newtheorem{remark}[theorem]{Remark}

\newtheorem{corollary}[theorem]{Corollary}
\newtheorem{setup}[theorem]{Set-up}

\newtheorem*{notation*}{Notation}

\begin{document}

\title[Symbolic powers of vertex cover ideals]{Symbolic powers of vertex cover ideals}

\author[S Selvaraja]{S Selvaraja}
	\email{selva.y2s@gmail.com}
	\address{The Institute of Mathematical Sciences, CIT campus, Taramani, Chennai, INDIA - 600113}

\begin{abstract}
 Let $G$ be a finite simple graph and $J(G)$ denote 
 its cover ideal in a polynomial ring over a field $\mathbb{K}$.
 %In this paper, we show that all symbolic powers of cover ideals of various classes of vertex decomposable graphs have linear quotients.
 In this paper,  we show  that all  symbolic powers of cover ideals of certain vertex decomposable graphs have
 linear quotients. 
 Using these results, we give various conditions on a subset $S$ of the vertices of $G$ so that 
 all symbolic powers of vertex cover ideals of $G \cup W(S)$, 
 obtained from $G$ by adding a whisker to each vertex in $S$, have linear quotients.
 For instance, if $S$ is a vertex cover of $G$, then all symbolic powers of $J(G \cup W(S))$ have linear quotients. 
 %Also, we show that all  symbolic powers of $J(G)$ have linear quotients 
%when $G$ is either a Cameron-Walker graph or a clique whiskering of some graph.
Moreover, we compute the Castelnuovo-Mumford regularity  of symbolic powers of 
certain cover ideals.
 
 \end{abstract}
 \thanks{AMS Classification 2010: 13D02, 13F20}
 \maketitle
 \section{Introduction}
 Symbolic powers of ideals have been studied intensely over the last two decades. 
 We refer the reader to \cite{DDAGHN} for a review of results in the literature.
Let $R=\mathbb{K}[x_1,\ldots,x_n]$ be the polynomial ring in $n$ variables over a field $\mathbb{K}$ and 
 $I$ be a squarefree monomial ideal in $R$. 
 The \emph{$k$-th symbolic power} of $I$, denoted by $I^{(k)}$, is the intersection of all primary components of 
 $I^k$ corresponding to minimal primes of $I$ (cf. \cite[Proposition 1.4.4]{Herzog'sBook}).
 In this paper we study the symbolic powers of cover ideals of graphs. 
 
 Let $G=(V(G),E(G))$ denote a finite, simple (no loops, no multiple edges), undirected graph with 
 vertex set $V(G) =\{x_1,\ldots,x_n\}$ and edge set $E(G)$.
 A { \it minimal vertex cover} of $G$ is a subset $C \subseteq V(G)$ such that each edge has at least one 
 vertex in $C$ and no proper subset of $C$ has the same property. 
 For a graph $G$, by identifying the vertices with variables in 
$R=\mathbb{K}[x_1,\ldots,x_n]$,  we associate squarefree monomial ideals, 
the {\it edge ideal} $I(G)=\left( x_ix_j \mid \{x_i,x_j\} \in  E(G) \right) $ and the {\it cover ideal} 
$J(G)= \left( \prod\limits_{x \in C}x \mid C \text{ is a minimal vertex cover of } G \right).$ 
The cover ideal of a graph  $G$ is the Alexander dual of its edge ideal, i.e.,
$J(G)=I(G)^\vee=\bigcap\limits_{\{x_i, x_j\}\in E(G)} (x_i,x_j)$. 
Recently, a dictionary between various combinatorial data of the 
graph $G$ and the algebraic properties of corresponding ideals $I(G)$ and $J(G)$ 
has been set up by various authors
(cf. \cite{BFH15, CamWalker, Nursel, FH, fv, GRV05,Herzog'sBook, HerHibiOhsugi, Hibi_cameronwalker,Fatemesh11,
Fatemesh14, Seyed, Fakhari, vill_cohen, Wood2}).
 
 Let  $I$ be a monomial ideal in $R$.  A homogeneous ideal $I$ is called {\it componentwise linear} 
 if for each $\ell$, the ideal generated
 by all degree $\ell$ elements of $I$ has a linear resolution. 
 Monomial ideals which are componentwise linear were introduced in \cite{HerHibi} by Herzog and Hibi 
 and have  strong combinatorial implications \cite{Herzog'sBook}. 
 Ideals with linear quotients were defined by Herzog and Takayama \cite{HerTak}
 in connection to their work on minimal free resolution of monomial ideals.
 A monomial ideal $I \subset R$ is said to have \textit{linear quotients}
 if there is an ordering $u_1 < \cdots < u_m$ on the minimal monomial generators of $I$ 
 such that for every $2 \leq i \leq m $ the ideal 
 $(u_1,\ldots,u_{i-1}) : (u_i)$ is generated by a subset of $\{x_1,\ldots, x_n\}$.
 %This notion was introduced by  Herzog and Takayama in \cite{HerTak}.  
 If a monomial ideal has linear quotients, then it has componentwise linear quotients, \cite{JahanZheng}, 
 and hence it is componentwise linear. 
 Componentwise linear ideals have a number of algebraic and combinatorial properties that make them 
 interesting to study. 

A graph $G$ is said to be vertex decomposable if $\Delta(G)$ is a vertex decomposable, 
where $\Delta(G)$ denotes the independence complex of $G$ 
(see Section \ref{pre} for definition).
Vertex decomposability was first introduced by Provan and Billera \cite{ProvLouis},
in the case when all the maximal faces are of equal cardinality, and
extended to the arbitrary case by Bj\"orner and Wachs \cite{BjWachs}.
We have the chain of implications:
\begin{equation*}
 \text{vertex decomposable}\Longrightarrow \text{shellable} \Longrightarrow \text{sequentially
 Cohen-Macaulay,}
\end{equation*} 
where a graph $G$ is shellable if $\Delta(G)$ is a shellable
simplicial complex and $G$ is sequentially Cohen-Macaulay if $R/I(G)$
is sequentially Cohen-Macaulay. 
In \cite{eagon}, Eagon and Reiner proved that $G$ is Cohen-Macaulay if and only if 
$J(G)$ has a linear resolution. Thereafter, 
Herzog and Hibi \cite{HerHibi}  and Herzog, Reiner and Walker \cite{HerReiWelker}
 proved that 
 $G$ is sequentially Cohen-Macaulay  if and only if 
 $J(G)$  has componentwise linear. In general it is hard to prove that an ideal is componentwise linear. 
 We refer the reader to \cite{Herzog'sBook} and the references cited there for a review of results in the literature in this direction.
 %We refer the reader to  \cite{Herzog'sBook} for a review of results in the literature. 
 In \cite{fv}, Francisco and Van Tuyl  proved that 
 if $G$ is a chordal graph, then $G$ is sequentially Cohen-Macaulay
 and hence $J(G)$ has componentwise linear. %Now one can ask the following question:
%Given a (sequentially) Cohen-Macaulay graph, what can be said about the powers of its vertex cover ideal? 
   In \cite{HerHibiOhsugi}, Herzog, Hibi, and Ohsugi gave a condition on homogeneous ideals 
   having the property  that all their powers are componentwise linear. 
   They also conjectured that all powers of the vertex cover ideal of chordal graphs are componentwise linear.
   There have been some attempts on proving all powers of cover ideals have componentwise linear for the
   subclass of chordal graphs, see \cite{HerHibi, Fatemesh11}. 
   In \cite{Fatemesh14}, Fatemesh gave a combinatorial condition on a graph which guarantees that all 
   powers of its vertex cover ideal are componentwise linear. 
   Recently, Nursel \cite{Nursel} proved that if $G$ is a $(C_4,2K_2)$-free graph,
   then $J(G)^k$ has componentwise linear for all $k \geq 1$.

   In this context it is natural to ask what happens when
   we consider the symbolic powers of cover ideals? More precisely, 
   given a (sequentially) Cohen-Macaulay graph, what can be said about the symbolic powers of its vertex cover ideal?
   In general, if $G$ is sequentially Cohen-Macaulay, then symbolic powers of $J(G)$ need not have 
 componentwise linear (Example \ref{ex:compnt}). 
   Recently, Fakhari \cite{Fakhari} proved that  if $G$ is a 
   Cohen-Macaulay and very well-covered graph,
   then  $J(G)^{(k)}$ has  linear quotients and hence it is componentwise linear.
 
 There has been a lot of work on how the combinatorial
 modification of the graphs would affect algebraic properties 
 of its edge ideals/cover ideals.  Generally, this question is interesting in the sense that 
 we may start with a graph with ``bad'' 
 algebraic properties, but with a slight modification, the edge ideal/cover ideal becomes much nicer.
For example, Villarreal proved that if $G$
is any graph, then $W(G)$ is Cohen-Macaulay, where $W(G)$ is
the graph obtained by adding a whisker to each vertex of $G$, \cite{vill_cohen}. 
Dochtermann and Engstr\"om \cite{DochEng09}
and Woodroofe \cite{Wood2} independently showed that $W(G)$ is a vertex decomposable graph.
Cook and Nagel \cite{CookNagel} generalized the whiskered idea and constructed the vertex clique-whiskered graph
$G^{\pi}$, see Definition \ref{cook},  and proved that $G^{\pi}$ is unmixed and vertex decomposable.
In \cite{BFH15}, Biermann et al.,  gave sufficient conditions on $S \subset V(G)$ such
that $G \cup W(S)$ is vertex decomposable, where 
$G∪\cup W(S)$ is the graph obtained from $G$ by adding a whisker at each vertex in $S$ (see also \cite{FH}).
 Later, Hibi et al., \cite{Hibi_cameronwalker} 
gave a generalization of Villarreal's result by showing that the graph obtained by attaching 
a complete graph to each vertex of a graph $G$ is unmixed and vertex decomposable.
In a different direction, several authors have studied similar phenomena (cf. \cite{Haj_yessemi, AFY15}).

 We consider the graph obtained by attaching a connected graph
 to some of the vertices of a graph. %A \textit{star complete}, denoted by $\kk(x)$,  is a graph joining some complete graphs  at one common vertex $x$.
 Let $H$ be a graph and $\{x_{i_1},\ldots,x_{i_q}\} \subseteq V(H)$. The graph 
 $G$ is obtained from $H$ by attaching $\kk(x_{i_j})$ to $H$ at the $x_{i_j}$ for all
$1 \leq j \leq q$, where $\kk(x_{i_j})$ is a graph joining some complete graphs 
 at one common vertex $x_{i_j}$.
 In this paper, we prove that all symbolic powers of the vertex cover ideals of such graphs with additional hypothesis
 have linear quotients  and hence it is componentwise linear 
 (Theorem \ref{main-result2}, Theorem \ref{main-result1}). 
 The above results has a number of interesting consequence. For example, 
 Corollary \ref{CamWal-verdec} says that if $G$ is a Cameron-Walker graph, then $J(G)^{(k)}$ has  
 linear quotients for all $k \geq 1$. Also, we explore on how to add whiskers to a graph so that
 all symbolic powers of its cover ideal have linear quotients.
 We give some sufficient conditions on a subset $S$ of the vertices of $G$ so that 
 all symbolic powers of  $J(G \cup W(S))$  have linear quotients (Corollary \ref{main-cor}).   
 Let $G$ be any graph and $\pi$ be any clique vertex-partition of $G$. Also, we prove that
 all symbolic powers of  $J(G^{\pi})$ have
 linear quotients (Theorem \ref{main-clique}). 
 As an immediate consequence of the above results, 
 we compute the Castelnuovo-Mumford regularity of 
 symbolic powers of certain cover ideals (Corollary \ref{reg-cor}). % in terms of the maximum degree of the cover ideals.
 %obtain the linear polynomial
%corresponding to $\reg(J(G)^{(k)})$ for some classes of graphs,
%where $\reg(-)$ denote the Castelnuovo-Mumford regularity of $-$ 
%Finally, we compute the graded Betti numbers of 
%symbolic powers of cover ideals of complete graphs (Theorem \ref{betti-com})  

 Our paper is organized as follows. In Section \ref{pre}, we collect the necessary notation, terminology
and some results that are used in the rest of the paper.
We prove, in Section \ref{tech}, several technical lemmas which are needed for the proof of 
our main results which appear in Section \ref{main}.
  \section{Preliminaries}\label{pre}
 In this section, we set up the basic definitions and notation needed for the main results. 
For a graph $G$, $V(G)$ and $E(G)$ denote the set of all
vertices and the set of all edges of $G$ respectively.
For $\{u_1,\ldots,u_r\}  \subseteq V(G)$, let $N_G(u_1,\ldots,u_r) = \{v \in V (G)\mid \{u_i, v\} \in E(G)~ 
\text{for some $1 \leq i \leq r$}\}$ be the set of neighbors of $u_1,\ldots,u_r$
and $N_G[u_1,\ldots,u_r]= N_G(u_1,\ldots,u_r) \cup \{u_1,\ldots,u_r\}$. %and set $N_G[u]=N_G(u) \cup \{u\}$.
The cardinality of $N_G(u)$ is called the \textit{degree} of $u$ in $G$ and is denoted by
$\deg_G(u)$. A subgraph $H \subseteq G$  is called {\it induced} if for $u, v
\in V(H)$, $\{u,v\} \in
E(H)$ if and only if $\{u,v\} \in E(G)$.
A subset
$X$ of $V(G)$ is called {\it independent} if there is no edge $\{x,y\} \in E(G)$ 
for $x, y \in X$. 
A {\it complete graph} is a graph in which each pair of graph vertices is connected by an edge. 
 %The complete graph with $n$ vertices is denoted $K_n$.
 A subset $U$ of $V(G)$ is said to be a \textit{clique} if the induced subgraph with vertex set $U$
is a complete graph.  
%A vertex $v$ is said to be \textit{free vertex} if it belongs to exactly one
%maximal clique. 
A \textit{simplicial vertex} of a graph $G$ is a vertex $x$ such that the neighbors of $x$ form a clique in
$G$. %The vertex $x$ is a {\it simplicial vertex} of $G$ if $N_G[x]$ is a complete graph.  
Note that if  $\deg_G(x)=1$, 
then $x$ is a simplicial vertex of $G$.  
 
 \begin{remark}
 Let $G$ be a graph. For $U \subseteq V(G)$, define $G \setminus U$ to
be the induced subgraph of $G$ on the vertex set $V(G) \setminus U$.  
Let $H$ be a subgraph of $G$ and $X \subseteq V(H)$, $Y \subseteq V(G)\setminus V(H)$. For the 
 notation, we shall use $H \setminus \{X \cup Y\}$ to also refer to 
$H \setminus X$.
 \end{remark}

 Let $G$ be a graph. For $S \subseteq V(G)$, let $G \cup W(S)$ denote the graph on the 
vertex set $V(G) \cup \{z_x \mid x \in S\}$ whose edge set is
$E(G \cup W(S))=E(G) \cup \Big\{ \{x,z_x\} \mid x \in S\Big\}$. 
An edge of the form $\{x, y\}$, where $N_G(y) = \{x\}$ is
called a \emph{whisker} of $G$.
 
 We recall the relevant background on simplicial complexes.  
 A \textit{simplicial complex} $\Delta$ on $V = \{x_1,\ldots,x_n\}$ is
a collection of subsets of $V$ such that: 
\begin{enumerate}
 \item[(i)] $\{x_i\}\in \Delta $ for $1 \leq i\leq n$, and
 \item[(ii)] if $F \in \Delta$ and $F' \subseteq F$, then $F' \in \Delta$.
\end{enumerate}
Elements of $\Delta$ are called the \textit{faces} of $\Delta$, and the maximal elements, with 
respect to inclusion, are called the \textit{facets}. 
%A simplicial complex is \textit{pure} if all its facets have the same cardinality.
The link of a face $F$ in
$\Delta$ is $\link_\Delta(F) = \{F' \mid F' \cup F \text{ is a face in
} \Delta, ~F' \cap F = \emptyset \}.$

A simplicial complex $\Delta$ is
recursively defined to be {\em vertex decomposable} if it is either a
simplex or else has some vertex $v$ so that 
\begin{enumerate}
  \item[(i)] both $\Delta \setminus v$ and $\link_\Delta (v)$ are vertex decomposable, and
  \item[(ii)] no face of $\link_\Delta (v)$ is a facet of $\Delta \setminus v$.
\end{enumerate}

The \textit{independence complex} of $G$, denoted by $\Delta(G)$, is the simplicial
complex on $V(G)$ with face set 
$\Delta(G)=\{F \subseteq V(G) \mid F \text{ is an independent set of $G$ } \}.$
A graph $G$ is said to be \textit{vertex decomposable} if $\Delta(G)$ is a
vertex decomposable simplicial complex. 
In \cite{Wood2}, Woodroofe translated the notion of vertex decomposable
for graphs as follows.
\begin{definition}\label{def-vertdecom} \cite[Lemma 4]{Wood2}
 A graph $G$ is recursively defined to be vertex decomposable if 
 $G$ is totally disconnected (with no edges) or if 
 \begin{enumerate}
  \item there is a vertex $x$ in $G$ such that $G \setminus x$ and 
  $G \setminus N_G[x]$ are both vertex decomposable, and
  \item no independent set in $G\setminus N_G[x]$ is a maximal independent set 
  in $G\setminus x$.
 \end{enumerate}
 \end{definition}

A \textit{shedding vertex} of $x$ is any vertex which satisfies 
either $\deg_G(x)=0$ or Condition (2) of Definition \ref{def-vertdecom}.

Let $G$ and $H$ be graphs. 
If $G$ and $H$ disjoint graphs (i.e., $V (G) \cap V (H) = \emptyset)$, we denote the disjoint union of 
$G$ and $H$ by $G\coprod H$.

\begin{corollary}\cite[Corollary 7(1) and Lemma 20]{Wood2}\label{cor-shedding} \mbox{}
\begin{enumerate}
 \item Any neighbor of a simplicial vertex of $G$ is a shedding vertex of $G$.
 \item Let $G$ and $H$ be two graphs.
Then $G \coprod H$ is vertex decomposable if and only if $G$ and $H$ are vertex decomposable.
\end{enumerate}
\end{corollary}

 For a graph $G$, the simplicial complex $\Delta_G$ with complete subgraphs (cliques) of $G$ as its faces 
is called the \textit{clique complex} of $G$.
 \begin{definition}
 A star complete, denoted by $\kk(x)$,  is a graph joining some complete graphs  at one common vertex $x$.
 \end{definition}
 
 Let $F_1,\ldots ,F_t$ be the facets of $\Delta_{\kk(x)}$. If $|F_i| \geq 3$ for all $1 \leq i \leq t$, then 
 $\kk(x)$ is said to be   {\it pure star complete graph.}
  Otherwise,  $\kk(x)$ is  {\it non-pure star complete graph}.
The graph given below on the right is a pure star complete graph while the graph on the 
left is not a pure star complete graph 
as $\Delta_{\kk(a)}$ has a facet with cardinality 2.

  %\centering

%\captionsetup[figure]{labelformat=empty}
\begin{figure}[H]
\begin{tikzpicture}[scale=0.7]
\draw (1,4)-- (2,4);
\draw (1,4)-- (2,3);
\draw (2,4)-- (2,3);
\draw (3,4)-- (2,3);
\draw (4,4)-- (2,3);
\draw (6.34,4.02)-- (7.16,2.76);
\draw (7.34,4.02)-- (7.16,2.76);
\draw (6.34,4.02)-- (7.34,4.02);
\draw (7.16,2.76)-- (8.54,3.82);
\draw (10,4)-- (7.16,2.76);
\draw (8.48,4.7)-- (8.54,3.82);
\draw (8.48,4.7)-- (10,4);
\draw (8.54,3.82)-- (10,4);
\draw (8.48,4.7)-- (7.16,2.76);
\draw (-2.12,2.4) node[anchor=north west] {non-pure star complete};
\draw (6.32,2.4) node[anchor=north west] {pure star complete};
\begin{scriptsize}
\fill [color=black] (2,3) circle (1.5pt);
\draw[color=black] (2.26,2.96) node {$a$};
\fill [color=black] (1,4) circle (1.5pt);
\fill [color=black] (2,4) circle (1.5pt);
\fill [color=black] (3,4) circle (1.5pt);
\fill [color=black] (4,4) circle (1.5pt);
\fill [color=black] (7.16,2.76) circle (1.5pt);
\draw[color=black] (7.6,2.76) node {$b$};
\fill [color=black] (7.34,4.02) circle (1.5pt);
\fill [color=black] (6.34,4.02) circle (1.5pt);
\fill [color=black] (8.54,3.82) circle (1.5pt);
\fill [color=black] (10,4) circle (1.5pt);
\fill [color=black] (8.48,4.7) circle (1.5pt);
\end{scriptsize}
\end{tikzpicture}
\caption{star complete graph}
\end{figure}
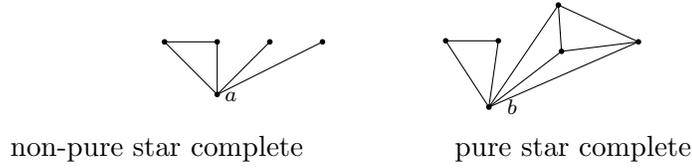

Let $M$ be a graded $R$ module. For non-negative
integers $i, j$, let $\beta_{i,j}(M)$ denote the $(i,j)$-th \emph{graded Betti
number} of $M$. The 
\emph{Castelnuovo-Mumford regularity} of $M$, denoted by $\reg(M)$, is defined as 
$\reg(M)=\max \{j-i \mid \beta_{i,j}(M) \neq 0\}$. 

Polarization is a process that creates a squarefree monomial ideal (in
a possibly different polynomial ring) from a given monomial ideal,
\cite[Section 1.6]{Herzog'sBook}. If $I$ is a monomial ideal in $R$, then the polarization of
$I$ denoted by $\widetilde{I} \subseteq \widetilde{R}$.
In this paper, we repeatedly use the following properties of the
polarization, namely:
\begin{corollary}\label{pol_reg}  Let $I$ be a monomial ideal in 
$R.$ Then
\begin{enumerate}
 \item \cite[Corollary 1.6.3]{Herzog'sBook} for 
 all $\ell,j$, $\beta_{\ell,j}(R/I)=\beta_{\ell,j}(\widetilde{R} / \widetilde {I})$.
 \item \cite[Lemma 3.5]{Fakhari} $I$ has linear quotients if and only if $\widetilde{I}$ has linear
quotients.
\end{enumerate}
\end{corollary}
 
In the study of symbolic powers of cover ideals, Fakhari constructed a new graph $G_k$ whose cover ideal is strongly related to the 
$k$-th symbolic power of cover ideal of $G$ \cite{Fakhari}.
This construction  has proved to be quite powerful, which we shall make use of often.

%This has emerged as a fine tool in the inductive process of computation of symbolic of powers of cover ideals.
%We recall the construction from \cite{Fakhari}.
\begin{cons}\label{construction}
 Let $G$ be a graph with vertex set $V(G)= \{x_1,\ldots,x_n\}$ and $k \geq 1$
be an integer. We define the new graph $G_k$ on new vertices
\[
 V(G_k)=\{x_{i,p} \mid 1 \leq i \leq n \text{ and } 1 \leq p \leq k\},
\]
and the edge set of $G_k$ is
\[
 E(G_k)=\Big\{\{x_{i,p},x_{j,q}\} \mid \{x_i,x_j\} \in E(G) \text{ and } p+q \leq k+1\Big\}.
\]
%Note that $(G \coprod H)_k=(G_k \coprod H_k)$, for any $k \geq 1$.
\end{cons}
Throughout this paper, $G_k$ denotes the graph as in Construction \ref{construction}.
The following observation is an immediate consequence of the construction:
\begin{obs}
If $G=H \coprod L$, then $G_k=H_k \coprod L_k$ for all $k \geq 1$.
\end{obs}

  The following lemma, due to Fakhari, is used repeatedly throughout this paper
\begin{lemma}\cite[Lemma 3.4]{Fakhari}\label{fak-result}
 Let $G$ be a graph. For every integer $k \geq 1$, the ideal $\widetilde{(J(G)^{(k)})}$ is the cover ideal of $G_k$.
\end{lemma}

 \section{Vertex decomposable graph}\label{tech}
 Our aim in this section is to prove that 
 $G_k$ is a vertex decomposable graph for all $k \geq 1$ when $G$ is a star complete graph.
 For this purpose, we need to get more details about the structure of the graph $G_k$. 
 As a first step towards this, we describe the simplicial vertices of $G_k$.
 
\begin{lemma}\label{simplicial-vertex}
Let $G$ be a graph. If $x_i$ is a simplicial vertex of $G$, then $x_{i,k}$ is a simplicial 
vertex of $G_k$ for all $k \geq 1$.
In particular, if $x_l \in N_G(x_i)$, then 
$x_{l,1}$ is a shedding vertex of $G_k$ for all $k \geq 1$.
\end{lemma}
\begin{proof}
 If $N_G(x_i)=\{x_{j_1},\ldots,x_{j_r}\}$, then by Construction \ref{construction}, 
 $N_{G_k}(x_{i,k})=\{x_{j_1,1},\ldots,x_{j_r,1}\}$.
 Since the neighbors of $x_i$ form a clique, the neighbors of $x_{i,k}$ form a clique.
  Hence $x_{i,k}$ is a simplicial vertex of $G_k$.
 By Corollary \ref{cor-shedding}(1), $x_{l,1}$ is a shedding vertex of $G_k$.
\end{proof}
The following lemmas further reveals the structure of $G_k$.
\begin{lemma}\label{technical-le}
 Let $H$ be a graph with vertices $V(H)=\{x_1,\ldots,x_n\}$. The graph $G=H(K_m)$ is obtained 
 from $H$ by attaching $K_m$ to $H$ at the vertex $x_1$, where $K_m$ is a complete
 graph with vertices $\{y_1=x_1,y_2,\ldots,y_m\}$. %Let $G_k$ be the graph 
 \begin{enumerate}
  \item If $m \geq 3$, then $y_{i,1}$ is a shedding vertex of $G_k \setminus \{y_{1,1},\ldots,y_{i-1,1}\}$
  for all $1 \leq i \leq m$.
  \item If $m=2$, then $y_{1,i}$ is a shedding vertex of $G_k \setminus \{y_{1,1},\ldots,y_{1,i-1}\}$ 
  for all $1 \leq i \leq k$.
 \end{enumerate}
\end{lemma}
\begin{proof}
 (1) Since $y_m$ is a simplicial vertex of $G$, by Lemma \ref{simplicial-vertex}, $y_{m,k}$ is a 
 simplicial vertex of $G_k$ and hence by Corollary \ref{cor-shedding}(1), $y_{i,1}$ is a
 shedding vertex of $G_k \setminus \{y_{1,1},\ldots,y_{i-1,1}\}$ for all $1 \leq i \leq m-1$.
 Note that $N_{G_k}(y_{j,k})=\{y_{1,1},\ldots,y_{m,1}\}$ and 
 $\deg_{G_k \setminus \{y_{1,1},\ldots,y_{m-1,1}\}}
 (y_{j,k})=1$ for all $2 \leq j \leq m-1$.  Since $\{y_{m,1},y_{j,k}\} \in E(G_k)$, by Corollary \ref{cor-shedding}(1), 
 $y_{m,1}$ is a shedding vertex of  $G_k \setminus \{y_{1,1},\ldots,y_{m-1,1}\}$.
 \vskip 0.3mm \noindent
 (2)  Note that  $N_{G_k \setminus \{y_{1,1},\ldots,y_{1,i-1}\}}(y_{2,k-(i-1)})=\{y_{1,i}\}$ 
 for all $1 \leq i \leq k$.  Therefore, 
 $y_{2,k-(i-1)}$ is a simplicial vertex of $G_k \setminus \{y_{1,1},\ldots,y_{1,i-1}\}$. 
 Hence by Corollary \ref{cor-shedding}(1), $y_{1,i}$ is a shedding vertex of  $G_k \setminus \{y_{1,1},\ldots,y_{1,i-1}\}$.
 \end{proof}

 Two graphs $G=(V(G),E(G))$ and $H=(V(H),E(H))$ are said to be \textit{isomorphic} (and 
 written as $G \simeq H$) if and only if there exists a 1-1 and onto function 
 $\phi: V(G) \longrightarrow V(H)$ such that $\{u,v\} \in E(G)$ if and only if
 $\{\phi(u),\phi(v)\} \in E(H)$.

\begin{lemma}\label{tech_lemma-cons}
 Let $G$ be a graph with vertices $V(G)=\{x_1,\ldots,x_n\}$. Then for all $k \geq 2$,
 \begin{enumerate}
 \item for any $\{x_{i_1},\ldots,x_{i_r}\} \subseteq V(G)$, 
 $$G_k \setminus \{x_{i_1,1},\ldots,x_{i_1,k},\ldots,x_{i_r,1},\ldots,x_{i_r,k}\} :=
 (G \setminus \{x_{i_1},\ldots,x_{i_r}\})_k.$$
  \item $ G_k \setminus \{x_{1,1},\ldots,x_{n,1}\} \simeq G_{k-2} \cup \{x_{1,k},\ldots,x_{n,k}\},$ 
where $G_0$ is the empty graph.
\item  $G_k \setminus N_{G_k}[x_{j,1}] \simeq (G \setminus N_G[x_j])_k \cup \{x_{j,2},\ldots,x_{j,k}\} $
for all $1 \leq j \leq n$.
 \end{enumerate} 
\end{lemma}
\begin{proof}
(1) This follows directly from the Construction \ref{construction}.
\vskip 0.3mm \noindent
(2) Let $G'=G_k \setminus \{x_{1,1},\ldots,x_{n,1}\}$, 
 $G''=G_{k-2} \cup \{x_{1,k},\ldots,x_{n,k}\}$ and  $\Phi: V(G') \longrightarrow V(G'')$ be the map 
 defined by $$
  \Phi(x_{i,p}) =
  \begin{cases}
                                   x_{i,p-1} & \text{if $p \neq k$} \\
                                   x_{i,k} & \text{if $p=k$,} 
  \end{cases}
$$
for all $1 \leq i \leq n$, $2 \leq p \leq k$. 
 Clearly $|V(G')|=nk-n$  and $|V(G'')|=(k-2)n+n=nk-n$. Suppose $\{x_{i,p},x_{j,q}\} \in E(G')$
 for some $1 \leq i,j \leq n$, $2 \leq p,q \leq k$. Then $p+q \leq k+1$ if and only if $p-1+q-1 \leq k-1$.
 Hence $G'$ is isomorphic to $G''$.
 \vskip 0.3mm \noindent
 (3) Let
 $\Phi: V(G') \longrightarrow V(G'')$ be the map defined by $\Phi(x_{i,p})=x_{i,p}$,  where $G'=G_k \setminus N_{G_k}[x_{j,1}]$ and  
 $G''=(G \setminus N_G[x_j])_k \cup \{x_{j,2},\ldots,x_{j,k}\}$.  
 If $|N_G(x_i)|=t$, then  $|N_{G_k}(x_{i,1})|=kt$. 
 Therefore $|V(G')|=kn-(kt+1)=kn-kt-1$,
 $|V(G'')|=k(n-(t+1))+k-1=kn-kt-1$.
 Suppose $\{x_{a,b},x_{c,d}\} \in E(G')$. 
 Since $x_{a,b},x_{c,d} \notin N_{G_k}(x_{i,1})$, we get $x_a,x_b \notin N_G(x_i)$. Therefore 
 $\{x_{a,b},x_{c,d}\} \in E(G'')$. 
 If $\{x_{a,b},x_{c,d}\} \in E(G'')$, then one can similarly show that $\{x_{a,b},x_{c,d}\} \in E(G')$. 
 Hence $G'$ is isomorphic to $G''$. 
\end{proof}
 
 The following lemma is probably well-known. We include the proof for completeness.
 \begin{lemma}\label{obs-verdec}
 Let $G$ be a graph and $\{x_1,\ldots,x_m\} \subseteq V(G)$. Set $\Psi_0=G$,
 $$
  \Psi_i=\Psi_{i-1} \setminus x_i \text{ and } \Omega_i=\Psi_{i-1} \setminus N_{\Psi_{i-1}}[x_i] \text{ for all }
  1 \leq i \leq m.  
 $$
If 
\begin{enumerate}
 \item $x_i$ is a shedding vertex of $\Psi_{i-1}$, for all $1 \leq i \leq m$,
 \item $\Omega_i$ is a vertex decomposable graph, for all $1 \leq i \leq m$, and
 \item $\Psi_m$ is a vertex decomposable graph,
\end{enumerate}
then $G$ is a vertex decomposable graph.
\end{lemma}
\begin{proof}
Since $\Psi_m$ and $\Omega_m$ are vertex decomposable, 
$\Psi_{m-1}$ is vertex decomposable. But then because $\Psi_{m-1}$ and $\Omega_{m-1}$ are
vertex decomposable, so is $\Psi_{m-2}$, and so on. In particular, $\Psi_1$ and 
$\Omega_1$ are vertex decomposable. Hence $G$ is a vertex decomposable graph.
\end{proof}
\begin{corollary} \label{van-noupure}
 Let $G$ be a vertex decomposable graph and $x$ be a simplicial vertex of $G$. If
 $y \in N_G(x)$, then $G \setminus y$ is a vertex decomposable graph.
\end{corollary}
\begin{proof}
 Let $N_G(x)=\{y_1,\ldots,y_t\}$, where $t \geq 1$. It is enough to prove that $G \setminus y_1$ is vertex decomposable.
Set $\Psi_1=G\setminus y_1$,
 $
  \Psi_i=\Psi_{i-1}\setminus y_i \text{ and } \Omega_i=\Psi_{i-1} \setminus N_{\Psi_{i-1}}[y_i] \text{ for 
  all $2 \leq i \leq t$.}
$
By Corollary \ref{cor-shedding}(1), $y_i$ is a shedding vertex of $\Psi_{i-1}$ for all $2 \leq i \leq t$.
Since $x$ is a simplicial vertex of $G$, $\Omega_i=G \setminus N_G[y_i]$ for all $2 \leq i \leq t$ and
$G \setminus N_G[x]=\Psi_t \cup \{x\}$.
Therefore by \cite[Theorem 2.5]{BFH15} and Corollary \ref{cor-shedding}(2), 
 $\Psi_t$ and $\Omega_i$  are vertex decomposable  graphs for all $2 \leq i \leq t$.
Hence, by Lemma \ref{obs-verdec}, $G \setminus y_1$ is a vertex decomposable graph.
\end{proof}

The following result is crucial in obtaining our main results.

\begin{theorem}\label{complete-vedeco}
 If $G$ is a complete graph, then $G_k$ is a vertex 
 decomposable graph for all $k \geq1$.
\end{theorem}
\begin{proof}
  Let $V(G)=\{x_1,\ldots,x_n\}$. We prove the assertion by induction on $k$. If $k=1$, then by \cite[Corollary 7(2)]{Wood2}, 
 $G_1$ is vertex decomposable. Assume that $k \geq 2$. Set $\Psi_0=G_k$,
 \[
  \Psi_i=\Psi_{i-1} \setminus x_{i,1}\text{ and } \Omega_i=\Psi_{i-1}\setminus N_{\Psi_{i-1}}[x_{i,1}] \text{ for all
  } 1 \leq i \leq n.
 \]
By Lemma \ref{tech_lemma-cons}(2), (3), we have
$\Psi_n \simeq  G_{k-2} \cup \{x_{1,k},\ldots,x_{n,k}\}$ and 
$\Omega_i$ is totally disconnected for all $1 \leq i \leq n$. 
By induction on $k$ and Corollary \ref{cor-shedding}(2), 
$\Psi_n$ and $\Omega_i$ are vertex decomposable graphs for all $1 \leq i \leq n$. 
If $n=2$, then $N_{\Psi_0}(x_{2,k})=\{x_{1,1}\}$ and $N_{\Psi_1}(x_{1,k})=\{x_{2,1}\}$.
Therefore, 
by Corollary \ref{cor-shedding}(1), $x_{l,1}$ is a shedding vertex of $\Psi_{l-1}$ for all $1 \leq l \leq 2$.
If $n \geq 3$, then by Lemma \ref{technical-le}, 
$x_{i,1}$ is a shedding vertex of $\Psi_{i-1}$  for all $1 \leq i \leq n$.
Therefore, by Lemma \ref{obs-verdec}, $G_k$ is a vertex decomposable
graph.
\end{proof}

Let $G$ be a graph. If $Q=\{u_1,\ldots,u_r\} \subseteq V(G)$, then set 
$[Q]=\{u_{1,1},\ldots,u_{r,1}\} \subseteq V(G_k)$.
%A \textit{free vertex}  of a facet $F$
%is a vertex that does not belong to another facet.
We are now ready to state our main result of this section.
%We can extend Theorem \ref{complete-vedeco} slightly.
\begin{theorem}\label{star-verdec}
 If $G$ is a star complete graph, then $G_k$ is a vertex decomposable graph for all $k \geq 1$.
\end{theorem}
\begin{proof}
Let $G=\kk(x_1)$ and $F_1,\ldots,F_t$ be the facets of $\Delta_{G}$. We split the proof into two cases.
\vskip 0.1mm \noindent
 \textsc{Case 1:} Suppose $\kk(x_1)$ is a non-pure star complete graph. Set 
 %$V(\kk(x_1))=\{x_1,y_1,\ldots,y_m\}$,
 $\Psi_0=G_k$,
 $
  \Psi_i=\Psi_{i-1} \setminus x_{1,i}\text{ and } \Omega_i=\Psi_{i-1}\setminus N_{\Psi_{i-1}}[x_{1,i}] \text{ for all
  } 1 \leq i \leq k.
 $
It follows from  Lemmas \ref{technical-le}(2), \ref{tech_lemma-cons} that
 $x_{1,i}$ is a shedding vertex of $\Psi_{i-1}$ for all $1 \leq i \leq k$ and
$\Psi_k=(G \setminus \{x_1\})_k= (H_1)_k \coprod \cdots \coprod (H_t)_k$, where $H_i$ is a complete graph
with vertex set $F_i \setminus \{x_1\}$ for all $1 \leq i \leq t$. Therefore,
by Theorem \ref{complete-vedeco}, $\Psi_k$ is a vertex decomposable graph.
By Lemma \ref{obs-verdec}, it remains to prove that $\Omega_i$ is a vertex decomposable
graph for all $1 \leq i \leq k$. Note that $\Omega_1=G_k \setminus N_{G_k}[x_{1,1}]$ is totally disconnected and
hence $\Omega_1$ is vertex decomposable. 
Let $V(G)=\{x_1,\ldots,x_m\}$. 
For any $1 \leq i \leq k$, we have
$N_{G_k}(x_{1,i})=\{x_{2,1},\ldots,x_{2,k-i+1},\ldots,x_{m,1},\ldots,x_{m, k-i+1}\}$. 
Therefore for all $1 < i \leq k$,
\begin{align*}
 V(\Omega_i)&=
 \{x_{2,k-i+2},\ldots, x_{2,k},  \ldots, x_{m,k-i+2},\ldots,x_{m,k},x_{1,i+1},\ldots,x_{1,k}\}.
\end{align*}
Note that $\{x_{1,i+1},\ldots,x_{1,k}\}$ are isolated vertices in $\Omega_i$.  
Set $\Omega_i'=\Omega_i \setminus \{x_{1,i+1},\ldots,x_{1,k}\}$ for all $1 \leq i \leq k$.
By Corollary \ref{cor-shedding}(2), it is enough to prove that $\Omega_i'$ is a vertex decomposable graph for all $1 \leq i \leq k$.
If $G \setminus x_1$ is totally disconnected, then so is $\Omega_i'$ for all $1 < i \leq k$. 
Therefore $\Omega_i'$ is  vertex decomposable for all $1 \leq i \leq k$. 

Suppose 
$G \setminus x_1$ is not totally disconnected. If $1 < i<\lceil \frac{k+3}{2}\rceil$, 
then $\Omega_i'$ is totally disconnected.
%then $$x_{2,k-i+2},\ldots, x_{2,k},  \ldots, x_{m,k-i+2},\ldots,x_{m,k},x_{1,i+1},\ldots,x_{1,k}$$
%are isolated vertices. 
Therefore $\Omega_i'$ is  vertex decomposable for all $1 < i <\lceil \frac{k+3}{2}\rceil$. 
Assume that $i \geq \lceil \frac{k+3}{2}\rceil$. Now we claim that
$
 \Omega_i' \simeq (G \setminus x_1)_{2(i-1)-k} \bigcup \{ \text{ isolated vertices}\} 
 \text{ for all $i \geq \lceil \frac{k+3}{2}\rceil$}.
$
Let $\Phi: V(\Omega_i') \to V((G \setminus x_1)_{2(i-1)-k}) \cup \{\text{isolated vertices}\}$ be the map 
defined by
\begin{align*}
\Phi_i(x_{p,k-i+1+q})&=x_{p,q} \text{ for all } 2 \leq p \leq m,~1 \leq q \leq 2(i-1)-k; \\
\Phi(x_{p,k-i+1+q})&= \{\text{some isolated vertex}\} \text{ for all } 2 \leq p \leq m,~ 
2(i-1)-k+1 \leq q \leq i-1.
%\Phi(x_{1,i+q})&=\{\text{some isolated vertex}\} \text{ for all } 1 \leq q \leq k-i.
\end{align*}
Let $\{x_{p,k-i+1+q},x_{p',k-i+1+q'} \} \in E(\Omega_i')$ 
$\text{ for any } 2 \leq p,p' \leq m, 1 \leq q,q' \leq 2(i-1)-k$.
Then $2(k-i+1)+q+q' \leq k+1$ if and only if $q+q' \leq 2(i-1)-k+1$. Therefore 
$\Omega_i' \simeq (G \setminus x_1)_{2(i-1)-k} \bigcup \{ \text{ isolated vertices}\}$
for all $i \geq \lceil \frac{k+3}{2}\rceil$. 
Since $G \setminus x_1$ is the disjoint union of complete graphs, by Theorem \ref{complete-vedeco} and 
Corollary \ref{cor-shedding}, 
$\Omega_i'$ is a vertex decomposable graph.

 \vskip 0.3mm \noindent 
 \textsc{Case 2:} Suppose $\kk(x_1)$ is pure star complete.
  We prove by induction on $k$. If $k=1$, then by \cite[Corollary 7(2)]{Wood2}, 
 $G$ is a vertex decomposable graph. Assume that $k \geq 2$.
 Let $F_i=\{x_1,x_{i_1},\ldots,x_{i_{r_i}}\}$ for all $1 \leq i \leq t$.  
 Note that $x_{1,1}$ is a shedding vertex of $G_k$ and 
  $G_k \setminus N_{G_k}[x_{1,1}]$ is totally disconnected and hence 
 it is a vertex decomposable graph.
  It follows from  Lemmas \ref{technical-le}, \ref{tech_lemma-cons} that
 $G_k \setminus \{[F_1],\ldots,[F_t]\}
 =G_{k-2} \cup \{\text{isolated vertices}\}$ and 
 $x_{i_j,1}$ is a shedding vertex of 
 $G_k \setminus \mathcal{S}$ for any $1 \leq i \leq t$, $1 \leq j \leq r_i$, where 
 $\mathcal{S}=\{x_{1,1}, [F_1] \setminus \{x_{1,1}\},\ldots,[F_{i-1}] \setminus 
 \{x_{1,1}\},x_{i_1,1},\ldots,x_{i_{j-1},1}\}$.
  By induction on $k$, $G_k \setminus \{[F_1],\ldots,[F_t]\}$ 
 is vertex decomposable. 
  If $x_{i_j} \in F_i \setminus \{x_{1}\}$, then by Lemma \ref{tech_lemma-cons},
 $
  G_k \setminus \{N_{G_k}[x_{i_j,1}], \mathcal{S}\}= 
  \mathcal{H}\setminus \bigg\{[F_1] \setminus \{x_{1,1}\},\ldots,[F_{i-1}] \setminus 
 \{x_{1,1}\}\bigg\},    
 $
where $\mathcal{H}=(H_1)_k\coprod \cdots \coprod (H_{i-1})_k \coprod
  (H_{i+1})_k \cdots \coprod (H_{t})_k$ and $H_{l}$ is a complete graph with vertex set 
  $F_l \setminus x_1$ for all $1 \leq l \neq i \leq t$. 
By Theorem \ref{complete-vedeco} and Corollary \ref{cor-shedding}(2), $\mathcal{H}$ is a vertex decomposable graph. 
  Since $[F_1] \setminus \{x_{1,1}\},\ldots,[F_{i-1}] \setminus 
 \{x_{1,1}\}$ are neighbors of simplicial vertices of $\mathcal{H}$, by Corollary \ref{van-noupure},
  $G_k \setminus \{N_{G_k}[x_{i_j,1}], \mathcal{S}\}$ is vertex decomposable.
 Hence, by Lemma \ref{obs-verdec}, $G_k$ is a vertex decomposable graph.
 \end{proof}

We can extend Theorem \ref{star-verdec} slightly.
\begin{lemma}\label{del-verdec}
 Let $G=\kk(y_1)$ be a star complete graph and 
 $A_i=G_k \setminus \{y_{1,1},\ldots,y_{1, i}\}$ for all $1 \leq i \leq k$. 
 Then $A_i$ is a vertex decomposable graph for all $1 \leq i \leq k$.
\end{lemma}
\begin{proof}
Let $V(G)=\{x_1,\ldots,x_n,y_1\}$.   
First we claim that
$G'=A_i \setminus \{x_{1,1},\ldots,x_{n,1}\} \simeq G''= G_{k-2} \setminus \{y_{1,1},\ldots,y_{1,i-1}\}
  \bigcup \{x_{1,k},\ldots,x_{n,k},y_{1,k}\}$ for  $1 \leq i \leq k$.
   Let $\Phi: V(G') \longrightarrow V(G'')$ be the map  
 defined by 
 \begin{align*}
   \Phi(y_{1,p}) =
  \begin{cases}
                                   y_{1,p-1} & \text{if $p \neq k$} \\
                                   y_{1,k} & \text{if $p=k$} 
  \end{cases}
  \text{ and } \Phi(x_{i,p}) =
  \begin{cases}
                                   x_{i,p-1} & \text{if $p \neq k$} \\
                                   x_{i,k} & \text{if $p=k$,} 
  \end{cases}
 \end{align*} 
 Clearly $|V(G')|=k(n+1)-n-i$ and $|V(G'')|=(k-2)(n+1)+n+1-i+1=k(n+1)-n-i$.
 Proceeding as in the proof of Lemma \ref{tech_lemma-cons}, we can conclude that
 $G' \simeq G''$. Hence the claim. 
 
If $G$ is a non-pure star complete graph, then by \textsc{Case 1} of Theorem \ref{star-verdec},
$A_i$ is a vertex decomposable graph for all $1 \leq i \leq k$. Suppose $G$ is  pure star complete. 
%Now we prove that $A_i$  is a vertex decomposable graph for $1 \leq i \leq k$.
We prove the assertion by induction on $k$. If $k=1$, then $A_1=G_1 \setminus \{y_{1,1}\}$ is a chordal
graph and hence, by \cite[Corollary 7(2)]{Wood2}, $A_1$ is a vertex decomposable graph. Suppose 
$k>1$. Assume by induction that for any pure star complete graph $G$ and for any $1 \leq i \leq k-1$,
$A_i$ is a vertex decomposable graph. 
If $i=1$, then by claim $A_i \setminus \{x_{1,1},\ldots,x_{n,1}\} \simeq 
G_{k-2}  \cup \{\text{isolated vertices}\}$. By Theorem \ref{star-verdec} and Corollary \ref{cor-shedding}(2),
$A_i \setminus \{x_{1,1},\ldots,x_{n,1}\}$ is a vertex decomposable graph.
If $i>1$, then by claim $A_i \setminus \{x_{1,1},\ldots,x_{n,1}\} \simeq G_{k-2} \setminus \{y_{1,1},\ldots,y_{1,i-1}\} \cup 
\{\text{ isolated vertices}\}$. By induction on $k$, $A_i \setminus \{x_{1,1},\ldots,x_{n,1}\}$
is a vertex decomposable graph.
It follows from Lemma \ref{technical-le} that $x_{j,1}$ is a shedding vertex of $A_i \setminus \{x_{1,1},\ldots,x_{j-1,1}\}$.
Proceeding as in Theorem \ref{star-verdec}, one can show that 
 $A_i \setminus \{x_{1,1},\ldots,x_{l-1,1}, N_{A_i}[x_{l,1}]\}$ is vertex decomposable for any 
 $1 \leq i \leq n$.  Therefore, 
by Lemma \ref{obs-verdec}, $A_i$ is a vertex decomposable graph.
\end{proof}

\section{Linear quotients}\label{main}
In this section, we prove that symbolic powers of cover ideals
of certain vertex decomposable graphs have linear quotients.
We begin by fixing the notation which will be used for the rest of the section.
\begin{setup}\label{setup}
 Let $H$ be a graph and $\{x_{i_1},\ldots,x_{i_q}\} \subseteq V(H)$. The graph 
 $$G=H(\kk(x_{i_1}), \ldots, \kk(x_{i_q}))$$ is 
obtained from $H$ by attaching star complete $\kk(x_{i_j})$ to $H$ at the $x_{i_j}$ for all
$1 \leq j \leq q$.
\end{setup}
We are now ready to prove our first main result.

\begin{theorem}\label{main-result2}
Let $G$ be a graph as in Set-up \ref{setup}. Suppose $\kk(x_{i_j})$ is a pure star complete graph for all 
$1 \leq j \leq q$. If $|V(H)|-1\leq q$, then $J(G)^{(k)}$ has linear
 quotients for all $k \geq 1$.
\end{theorem}
\begin{proof}
 Let $V(H)=\{x_1,\ldots,x_n\}$.
 First we claim that $G_k$ is a vertex decomposable graph for all $k \geq 1$. 
  We prove this by induction on $m:=k+n$. 
 If $k \geq 1$ and $n=1$, then by Theorem \ref{star-verdec}, $G_k$ is  vertex decomposable.
 If $k=1$ and $n \geq 2$, then by \cite[Corollary 3.6(i)]{Haj_yessemi}, $G_1$ is vertex decomposable.
 
 Assume that $n>1$, $k > 2$.  
 If either $q=n$ or $q=n-1$, then by Lemma \ref{tech_lemma-cons}, either
 \begin{align*}
   G_k \setminus \{[V(\kk(x_1))],\ldots, [V(\kk(x_n))]\} & \simeq G_{k-2} \cup \{\text{isolated vertices}\} \text{ or }\\
   G_k \setminus \{[V(\kk(x_1))],\ldots, [V(\kk(x_q))], x_{n,1}\} & \simeq G_{k-2} \cup \{\text{isolated vertices}\}
   \text{ respectively}. 
 \end{align*}
Therefore by induction on $m$, $G_{k-2}$ is vertex decomposable.  
 Let $F_{t_1},\ldots,F_{t_{r_t}}$ be the facets of $\Delta_{\kk(x_t)}$ for all $1 \leq t \leq q$.  
For $1 \leq t \leq n$, $1 \leq j \leq r_t$, set $\A_t=\bigcup\limits_{1 \leq a \leq t-1}[V(\kk(x_a))]$, 
$\B_{t_j}=\bigcup\limits_{1 \leq b \leq j-1}([F_{t_b}]\setminus x_{t,1}).$  Let $x_{t_{j_l},1} \in F_{t_j}$ for some $1 \leq t \leq q$, $1 \leq j \leq r_t$.
It follows from Lemma \ref{technical-le} that
$x_{t_{j_l},1}$ is a shedding vertex of $G_k \setminus \{\A_t,x_{t,1},\B_{t_j},x_{t_{j_1},1},\ldots,x_{t_{j_{l-1}},1}\}$, 
where 
 $x_{t_{j_1},1},\ldots,x_{t_{j_{l-1}},1} \in [F_{t_j}]$. 
If $q=n$, then by Lemma \ref{technical-le}, $x_{t,1}$ is a shedding vertex of 
$G_k \setminus \A_t$ for all $1 \leq t \leq n$. 
Suppose $q=n-1$.
By Lemma \ref{technical-le}, $x_{t,1}$ is a shedding vertex of 
$G_k \setminus \A_t$ for any $1 \leq t \leq q$. If $\deg_G(x_{n})=0$, then $x_{n,1}$ is a 
shedding vertex of $G_k \setminus \A_n$. Suppose $\deg_G(x_n)>0$. Let $z_i \in N_{G}(x_n)$. Note that
$\{z_{i,k},x_{n,1}\}\in E(G_k)$ and $\deg_{G_k \setminus \A_n}(z_{i,k})=1$. Therefore, by Corollary 
\ref{cor-shedding}(1), $x_{n,1}$ is a shedding vertex of $G_k \setminus \A_n$.

Suppose $q=n$. Set $ N_H(x_t) =\{x_{t_1},\ldots,x_{t_{\alpha_t}}\},$
$V(H)\setminus N_H[x_t]=\{x_{t'_1},\ldots,x_{t'_{\beta_{t'}}}\}$,
 $G'=\left(H\setminus\{x_{t_1},\ldots,x_{t_{\alpha_t}}\}\right)(\kk(x_{t'_1}),\ldots, \kk(x_{t'_{\beta_{t'}}}))$ and
 $$G''=\left(H \setminus x_t\right)(\kk(x_1),\ldots,\kk(x_{t-1}),\kk(x_{t+1}),\ldots,\kk(x_n)).$$
%\begin{align*}
 
%\end{align*}
By Lemma \ref{tech_lemma-cons}, we have for any $1 \leq t \leq n$, $1 \leq j \leq r_t$, 
\begin{align}
  G_k \setminus \big\{N_{G_k}[x_{t,1}], \A_t\big\} \simeq  \bigg( (G')_k   \coprod_{i=1}^{\alpha_t}
   (\kk(x_{t_i})\setminus x_{t_i})_k \bigg)\setminus \A_t  \cup \{\text{isolated vertices}\}& \label{first:eq1} \\
G_k \setminus \big\{N_{G_k}[x_{t_{j_l},1}], \A_t,x_{t,1},\B_{t_j},x_{t_{j_1},1},\ldots,
x_{t_{j_{l-1}},1}\big\} \simeq    \{\text{isolated vertices}\}&  \coprod \label{first:eq2}\\
(G'')_k \setminus \{\A_t\}\coprod \bigg( \coprod (L_1)_k \cdots \coprod (L_{j-1})_k \coprod 
(L_{j+1})_k \cdots \coprod (L_{r_t})_k\bigg)\setminus \{\B_{t_j}\}& \nonumber,
\end{align}
where 
 $L_i $  is the complete graph
with vertex set  $F_{t_{i}} \setminus \{x_t\}$ for all  $1 \leq i \neq j \leq r_t$. 
By induction on $m$, Theorem \ref{complete-vedeco} and Corollary \ref{cor-shedding},
$(G')_k   \coprod_{i=1}^{\alpha_t}(\kk(x_{t_i})\setminus x_{t_i})_k$ is a vertex decomposable
graph. Since every element in $\A_t$ is either a neighbor of simplicial vertex or
in  $(G')_k   \coprod_{i=1}^{\alpha_t}(\kk(x_{t_i})\setminus x_{t_i})_k$, by Corollary \ref{van-noupure},
we get \eqref{first:eq1} is a vertex decomposable graph. Similarly, we can show that
\eqref{first:eq2} is a vertex decomposable graph.
Suppose $q=n-1$. Now proceeding as in the above paragraph of the proof, one can show that
for any $1 \leq t \leq n$, $1 \leq j \leq r_t$, $G_k \setminus \big\{N_{G_k}[x_{t,1}], \A_t\big\}
\text{ and } G_k \setminus \big\{N_{G_k}[x_{t_{j_l},1}], \A_t,x_{t,1},\B_{t_j},x_{t_{j_1},1},\ldots,
x_{t_{j_{l-1}},1}\big\}$ are vertex decomposable graphs. Therefore, by Lemma \ref{obs-verdec},
$G_k$ is vertex decomposable for all $k \geq 2$. Hence the claim.  

By \cite[Proposition 8.2.5]{Herzog'sBook}, $J(G_k)$ has linear quotients for all $k \geq 1$.
Therefore, by  Lemma \ref{fak-result}, $\widetilde{ J(G)^{(k)}}$ has linear quotients. Hence by 
 Corollary \ref{pol_reg}, $J(G)^{(k)}$ has linear quotients for all $k \geq 1$.
\end{proof}

Using techniques similar to the ones used in the proof of Theorem \ref{main-result2}, Theorem \ref{star-verdec}
and Lemma \ref{del-verdec}, 
we prove our next main result.

\begin{theorem}\label{main-result1}
 Let $G$ be a graph as in Set-up \ref{setup}. Suppose $\kk(x_{i_1}),\ldots,\kk(x_{i_p})$
 are non-pure star complete graphs and $\kk(x_{i_{p+1}}),\ldots,\kk(x_{i_q})$
 are pure star complete graphs for some $p \leq q$. If $V(H) \setminus \{x_{i_{1}},\ldots,x_{i_p}\}$ is an 
 independent set of $H$, then $J(G)^{(k)}$ has linear
 quotients for all $k \geq 1$.
\end{theorem}
\begin{proof}
 Let $V(H)=\{x_1,\ldots,x_n\}$.
 By proof of  Theorem \ref{main-result2}, it is enough to prove that $G_k$ is vertex decomposable for all $k \geq 1$.
 We prove this by induction on $n$. If $n=1$ (since empty set is an independent set),
 then by Theorem \ref{complete-vedeco},
 $G_k$ is vertex decomposable for all $k \geq 1$.   
 
 Assume that $n \geq 2$. For a fixed $1 \leq l \leq k$, set $z_{(l-1)p+j}=x_{i_j,l}$ for all $1 \leq j \leq p$.
 Let $\Psi_0=G_k$,
 $\Psi_i =\Psi_{i-1}\setminus z_i \text{ and }\Omega_i=\Psi_{i-1} \setminus N_{\Psi_{i-1}}[z_i] \text{ for all }
  1 \leq i \leq kp.$ It follows from  Lemmas \ref{technical-le}, \ref{tech_lemma-cons} that
 $z_i$ is a shedding vertex of $\Psi_{i-1}$ for all $1 \leq i \leq kp$ 
and $\Psi_{kp}=(G \setminus \A)_k$, where $\A=\{x_{i_1},\ldots,x_{i_p}\}$.
Since  $V(H) \setminus \A$ is an 
 independent set of $H$,  $G \setminus \A$ is the disjoint union of pure
 star complete graphs and isolated vertices. By Theorem \ref{star-verdec} and Corollary \ref{cor-shedding}(2), 
 $\Psi_{kp}$ is a vertex decomposable graph. 
 It remains to show that $\Omega_i$ is  vertex decomposable for all
 $1 \leq i \leq kp$.  

 For some $1 \leq j \leq p$ and $1 \leq l \leq k$, let
 $$\Omega_{(l-1)p+j}=G_k \setminus \{ N_{G_k}[x_{i_j,l}],z_1,\ldots, z_{(l-1)p+j-1}\}=\G.$$  
 % Set $V(\kk(x_{i_j}))=\{x_{i_j},y_1,\ldots,y_{\mu}\}$ and $\G'=\G \setminus 
  %\{y_{1,1},\ldots,y_{1,k},\ldots,y_{\mu,1},\ldots,y_{\mu,k}\}$.  
  Let  $V(\kk(x_{i_j}))=\{x_{i_j},y_1,\ldots,y_{\mu}\}$  
    %V(G) \setminus N_G[x_{i_j}]=\mathcal{T} 
   and  
    $N_G[x_{i_j}]=\{ V(\kk(x_{i_j})),a_1,\ldots,a_{\alpha},b_1,\ldots,b_{\beta},c_1,\ldots,
  c_{\gamma}\},$ % \mathcal{S}&=\{a_{1,1},\ldots,a_{1,k+1-l},\ldots,a_{\alpha,1},\ldots,a_{\alpha,k+1-l}\},  
  where
  $\{a_1,\ldots,a_{\alpha}\} \subseteq \A, 
 \{b_1,\ldots,b_{\beta}\}  \subseteq \B=\{x_{i_{p+1}},\ldots,x_{i_q}\}$  and 
 $\{c_1,\ldots,c_{\gamma}\}   \subseteq \C=V(H)\setminus
  \{\A \cup \B\}. $ 
  %Let $V(\kk(x_{i_j}))=\{x_{i_j},y_1,\ldots,y_{\mu}\}$.
  Note that %$\mathcal{S} \cap \{z_1,\ldots,z_{(l-1)p+j-1}\} \neq \emptyset$ and 
  $ N_{G_k}(x_{i_j,l})$
  $$=\bigg\{y_{r,l'}, a_{u,l'},b_{v,l'},c_{w,l'} \mid 1 \leq r \leq \mu, 1 \leq u \leq \alpha,~1 \leq v \leq \beta,~1 \leq w \leq \gamma,1 \leq l' 
  \leq k+1-l\bigg\}.$$
  \iffalse
    Therefore
  \begin{align*}\label{se:ex1}
V(\G)&=\{y_{r,l'}, a_{u,l'},b_{v,l'},c_{w,l'} \mid 1 \leq r \leq \mu, 1 \leq u \leq \alpha,
~1 \leq v \leq \beta,~1 \leq w \leq \gamma,\\
&k+2-l \leq l'   \leq k\} \cup (V((G\setminus N_G[x_{i_j}])_k)\setminus \{z_1,\ldots,z_{(l-1)p+j-1}\} ) \cup
\{x_{i_j,l+1},\ldots,x_{i_j,k}\}.\nonumber
\end{align*}
\fi
Since $\C$ is an independent set of $H$, 
  $N_H(b_1,\ldots,b_{\beta},c_1,\ldots,c_{\gamma}) \subseteq \A$. Suppose
  $\{c_{w,l'}, x_{i_\kappa,\mu}\} \in E(\G)$ for some $1 \leq w \leq \gamma$, 
 $1 \leq \kappa \leq p$. Since $k+2-l \leq l' \leq k$ and $l \leq \mu \leq k$, we get $l'+\mu >k+1$. This is a
  contradiction to $\{c_{w,l'}, x_{i_\kappa,\mu}\} \in E(\G)$. Therefore, 
  $\{c_{w,l'}, x_{i_\kappa,\mu}\} \notin E(\G)$ for all $1 \leq w \leq \gamma$,  $1 \leq \kappa \leq p$. Similarly, we can show that 
  $\{b_{v,l'}, x_{i_\kappa,\mu}\} \notin E(\G)$ for all $1 \leq v \leq \beta$,  $1 \leq \kappa \leq p$.
  
  Set 
  \begin{align*}      
      H'&=H \setminus \{x_{i_j},b_1,\ldots,b_{\beta},c_1,\ldots,c_{\gamma}\}, ~
  \B \setminus \{b_1,\ldots,b_{\beta}\}=
 \{x_{i'_{p'+1}},\ldots,x_{i'_{q'}}\},\\
 G'&=H'(\kk(x_{i_1}),\ldots,\kk(x_{i_{j-1}}),\kk(x_{i_{j+1}}),\ldots,\kk(x_{i_p}),\kk(x_{i'_{p'+1}}),\ldots,
 \kk(x_{i'_{q'}}))\text{ and }\\
 \mathcal{H}&=(G')_k \setminus \{z_1,\ldots,z_{(l-1)p+j-1},\mathcal{S}\} 
      \coprod\limits_{1 \leq t \leq \beta} \kk(b_t)_k \setminus \{b_{t,1},\ldots,b_{t, k+1-l}\}, 
      \text { where }\\
      \mathcal{S}&=\{a_{1,1},\ldots,a_{1,k+1-l},\ldots,a_{\alpha,1},\ldots,a_{\alpha,k+1-l}\}.
 \end{align*}
  Let $\mathcal{L}$ be the induced subgraph of $\G$ over the vertices 
  $\{y_{r,l'}\mid 1 \leq r \leq \mu, k-l+2 \leq l' \leq k\}$.   
 By above arguments we can conclude that
 $E(\G)=E(\mathcal{H}) \cup E(\mathcal{L})$.
  If $\kk(x_{i_j}) \setminus \{x_{i_j}\}$ is  totally disconnected, then
  $E(\G)=E(\mathcal{H})$.   
  Suppose $\kk(x_{i_j}) \setminus \{x_{i_j}\}$ is  not 
  totally disconnected. If either $l<\lceil \frac{k+3}{2} \rceil$ or $l \geq \lceil \frac{k+3}{2} \rceil$, then by \textsc{Case 1} of proof of Theorem
  \ref{star-verdec}, either $E(\G)=E(\mathcal{H}) \text{ or } 
  \G \simeq \mathcal{H}
   \coprod (\kk(x_{i_j})\setminus \{x_{i_j}\})_{2(l-1)-k} \cup \{\text{isolated vertices} \}.$
 
 By Theorem \ref{complete-vedeco}, Lemma \ref{del-verdec} and Corollary \ref{cor-shedding}(2),
  $(\kk(x_{i_j})\setminus \{x_{i_j}\})_{2(l-1)-k}$ and 
  $\coprod\limits_{1 \leq t\leq \beta} \kk(b_t)_k \setminus \Big\{b_{t,1},\ldots,b_{t, k+1-l}\Big\} $ are 
  vertex decomposable graphs. By Corollary \ref{cor-shedding}(2), it suffices to show that 
  $(G')_k \setminus \{z_1,\ldots,z_{(l-1)p+j-1},\mathcal{S}\}$ is a vertex decomposable
  graph. In order to achieve this, we consider the following two cases. 
  \vskip 0.3mm \noindent
 \textsc{Case I:} Suppose $\{x_{i_1},\ldots,x_{i_{j-1}},x_{i_{j+1}},\ldots,x_{i_p}\} \cap V(H') = \emptyset$.
 If $\{x_{i'_{p'+1}},\ldots,x_{i'_{q'}}\} \cap V(H')=\emptyset$, then $H'$ is totally disconnected. 
 Therefore $(G')_k \setminus \{z_1,\ldots,z_{(l-1)p+j-1},\mathcal{S}\}$ is a vertex decomposable
 graph. Suppose $\{x_{i'_{p'+1}},\ldots,x_{i'_{q'}}\} \cap V(H')\neq \emptyset$.
 Since $\C$ is an independent set of $H$, we have
 $H'$ is a disjoint union of isolated vertices. % we have $\mid \{x_{i'_{p'+1}},\ldots,x_{i'_{q'}}\} \cap V(H') \mid=1$.
 Hence $G'$ is the disjoint union of pure star complete graphs and isolated vertices.
 By Theorem \ref{star-verdec} and Corollary \ref{cor-shedding}(2), $(G')_k$ is a vertex decomposable graph. 
 Hence 
 %By Lemma \ref{technical-le}, $z_i$ and $\mathcal{S}$ are shedding vertex of $(G')_k$. By Theorem \ref{van-noupure},
 $(G')_k \setminus \{z_1,\ldots,z_{(l-1)p+j-1},\mathcal{S}\}$ is a vertex decomposable
 graph.
 \vskip 0.3mm \noindent
 \textsc{Case II:} Suppose $\{x_{i_1},\ldots,x_{i_{j-1}},x_{i_{j+1}},\ldots,x_{i_p}\} \cap V(H') \neq \emptyset$.
 Clearly $$V(H') \setminus \{x_{i_1},\ldots,x_{i_{j-1}},x_{i_{j+1}},\ldots,x_{i_p}\}$$
 is an independent set of $H'$. By induction on $n$, $(G')_k$ is a vertex decomposable graph.
 By Lemma \ref{technical-le}, $z_1,\ldots,z_{(l-1)p+j-1}$ and $\mathcal{S}$ are 
 neighbors of simplicial vertices. It follows from 
  Corollary \ref{van-noupure} that
 $(G')_k \setminus \{z_1,\ldots,z_{(l-1)p+j-1},\mathcal{S}\}$ is a vertex decomposable
 graph.  
 
 In both cases, we get $(G')_k \setminus \{z_1,\ldots,z_{(l-1)p+j-1},\mathcal{S}\}$ is a 
 vertex decomposable. %By Lemma \ref{del-verdec}, 
 Hence $\Omega_i$ is a vertex decomposable
 graph for all $1 \leq i \leq kp$.
\end{proof}
 
 The following example shows that the hypotheses of Theorem \ref{main-result2} and
 Theorem \ref{main-result1} can not easily be weakened.
 \begin{example}\label{ex:compnt}
  Let $G=H(\kk(x_1),\kk(x_3))$ and $L=L' \cup W(y_4)$ be the graphs as shown in figure, where 
  $$E(H)=\Big\{ \{x_1,x_2\}, \{x_2,x_3\},\{x_3,x_4\}, \{x_4,x_1\} \Big\},
  E(\kk(x_1))=\Big\{ \{x_1,x_5\},\{x_1,x_6\},\{x_5,x_6\} \Big\},$$ $E(\kk(x_3))=\Big\{\{x_3,x_7\},\{x_7,x_8\}
  ,\{x_3,x_8\} \Big\}$
  and 
  $L' =L \setminus y_6$. 

 \begin{minipage}{\linewidth}
  %\centering
\begin{minipage}{0.47\linewidth}
%\captionsetup[figure]{labelformat=empty}
\begin{figure}[H]
 \begin{tikzpicture}[scale=.5]
\draw (1,4)-- (3,4);
\draw (1,4)-- (1,3);
\draw (3,4)-- (3,3);
\draw (1,3)-- (3,3);
\draw (3,5)-- (4,4);
\draw (3,5)-- (3,4);
\draw (4,4)-- (3,4);
\draw (1,3)-- (0.47,1.99);
\draw (1,3)-- (1.74,2.01);
\draw (0.47,1.99)-- (1.74,2.01);
\draw (9,5)-- (8,4);
\draw (9,5)-- (10,4);
\draw (8,4)-- (10,4);
\draw (8,3)-- (10,4);
\draw (8,4)-- (10,3);
\draw (10,4)-- (10,3);
\draw (8,4)-- (8,3);
\draw (10,3)-- (10,2);
\draw (1.81,1.8) node[anchor=north west] {$G$};
\draw (8.8,2.03) node[anchor=north west] {$L$};
\begin{scriptsize}
\fill [color=black] (1,4) circle (1.5pt);
\draw[color=black] (1.25,4.23) node {$x_4$};
\fill [color=black] (3,4) circle (1.5pt);
\draw[color=black] (2.65,4.3) node {$x_1$};
\fill [color=black] (1,3) circle (1.5pt);
\draw[color=black] (0.55,3.23) node {$x_3$};
\fill [color=black] (3,3) circle (1.5pt);
\draw[color=black] (3.45,3.03) node {$x_2$};
\fill [color=black] (3,5) circle (1.5pt);
\draw[color=black] (3.25,5.24) node {$x_5$};
\fill [color=black] (4,4) circle (1.5pt);
\draw[color=black] (4.26,4.23) node {$x_6$};
\fill [color=black] (0.47,1.99) circle (1.5pt);
\draw[color=black] (0.0,1.9) node {$x_8$};
\fill [color=black] (1.74,2.01) circle (1.5pt);
\draw[color=black] (2.19,1.95) node {$x_7$};
\fill [color=black] (9,5) circle (1.5pt);
\draw[color=black] (9.25,5.34) node {$y_1$};
\fill [color=black] (8,4) circle (1.5pt);
\draw[color=black] (7.98,4.39) node {$y_3$};
\fill [color=black] (10,4) circle (1.5pt);
\draw[color=black] (10.35,4.23) node {$y_2$};
\fill [color=black] (8,3) circle (1.5pt);
\draw[color=black] (7.63,2.96) node {$y_5$};
\fill [color=black] (10,3) circle (1.5pt);
\draw[color=black] (10.35,3.23) node {$y_4$};
\fill [color=black] (10,2) circle (1.5pt);
\draw[color=black] (10.35,2.23) node {$y_6$};
\end{scriptsize}
\end{tikzpicture}
\caption{Vertex decomposable graphs}
\end{figure}
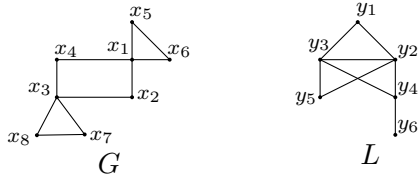
\end{minipage}
\begin{minipage}{0.49\linewidth}
By \cite[Corollary 3.7]{Haj_yessemi} and \cite[Corollary 7(2)]{Wood2}, 
$G$ and $L$ are vertex decomposable graphs respectively. 
Therefore, by \cite[Proposition 8.2.5]{Herzog'sBook} $J(G)$ and $J(L)$ has linear quotients. 
 It can also be noted that $G$ and $L$ does not satisfy the hypothesis of Theorem \ref{main-result2} and 
Theorem \ref{main-result1} respectively. 

\end{minipage}

A computation on \textsc{Macaulay2}, \cite{M2}, 
shows that $G_2$ is not sequentially Cohen-Macaulay.
By Lemma \ref{fak-result} and \cite[Theorem 2.1]{HerHibi}, $\widetilde{J(G)^{(2)}}$ is 
not a componentwise linear ideal.
%Let $L=L' \cup W(y_4)$ be the graph as shown in figure, where $L'=L \setminus y_6$.
%Since $L$ is a chordal graph, by \cite[Corollary 7(2)]{Wood2}, $L$ is a vertex decomposable graph.
\end{minipage}
 Therefore, by Corollary \ref{pol_reg},
$J(G)^{(2)}$ is not a componentwise linear ideal. Hence $J(G)^{(2)}$ has not linear quotients. 
Similarly we can show that
$J(L)^{(2)}$ has not linear quotients.
\end{example}
\iffalse
As an immediate consequence of Theorem \ref{star-verdec}, \cite[Proposition 8.2.5]{Herzog'sBook}, 
  Lemma \ref{fak-result} and Corollary \ref{pol_reg},  we have the following statement.
 \begin{corollary}\label{cor-star}
 If $G$ is a star complete graph, then for all $k \geq 1$, $J(G)^{(k)}$  has linear quotients.
\end{corollary}
\fi
The following is one of our main results which is a consequence of Theorem \ref{main-result1}.
\begin{corollary} \label{main-cor}
 Let $G$ be a graph and $S \subseteq V(G)$. 
 \begin{enumerate}
  \item If $G \setminus S$ is an independent set of $G$,
 then for all $k \geq 1$, $J(G \cup W(S))^{(k)}$  has linear quotients.
 \item If  $S$ is a vertex cover of $G$, then for all $k \geq 1$, $J(G \cup W(S))^{(k)}$  has linear quotients.
 \end{enumerate}
 \end{corollary}

 A \textit{matching} in a graph $G$ is a subgraph consisting of pairwise disjoint edges. The largest size of a 
  matching in $G$ is called its matching number and denoted by $\ma(G)$.
  If the subgraph is an induced subgraph, the matching is an \textit{induced matching}. The largest size of an induced matching in $G$ is called its induced 
matching number and denoted by $\nu(G)$.  A graph $G$ satisfies $\nu(G)=\ma(G)$ is called a \textit{ Cameron-Walker graph.}
 Cameron and Walker \cite{CamWalker} and Hibi et al.,\cite{Hibi_cameronwalker} gave a classification of the 
 connected graphs with $\nu(G)=\ma(G)$:
 \begin{theorem}\label{CamWal-char} \cite[Theorem 1]{CamWalker}, \cite[p. 258]{Hibi_cameronwalker}
 A connected graph $G$ is Cameron-Walker if and only if it is one of the following graphs:
 \begin{enumerate}
  \item a star;
  \item a star triangle (A star triangle is a graph consisting of some triangles joined at one common vertex);
  \item a finite graph consisting of a connected bipartite graph with bipartition $(A,B)$ such that there 
  is at least one leaf edge attached to each vertex $i \in A$ and that there may be possibly some pendant 
  triangles attached to each vertex  $j \in B$.
 \end{enumerate}
\end{theorem} 
%As an immediate consequence, we have the following statement.
We now study the $J(G)^{(k)}$ when $G$ is a Cameron-Walker graph.
\begin{corollary}\label{CamWal-verdec}
 If $G$ is a Cameron-Walker graph, then $J(G)^{(k)}$ has linear quotients for all $k \geq 1$.
\end{corollary}
\begin{proof}
 By Theorem \ref{CamWal-char}, we have
 $
  G= \coprod\limits_{1 \leq i\leq p} H_i  \coprod\limits_{1 \leq i \leq q} K_i  \coprod\limits_{1 \leq i \leq r} L_i,
$
where $H_i$'s are star graphs, $K_i$'s are star triangles and 
$L_i$'s are graphs as in Theorem \ref{CamWal-char}(3). Note that
 $L_i$ satisfies the hypothesis of Theorem \ref{main-result1} .
 By proof of Theorem \ref{main-result1}, Theorem \ref{star-verdec} and Corollary \ref{cor-shedding}, $G_k$ is a vertex decomposable graph for all $k \geq 1$.
 Hence, by \cite[Proposition 8.2.5]{Herzog'sBook} $J(G)^{(k)}$ has linear quotients for all $k \geq 1$. 
\end{proof}
\begin{definition}(Cook and Nagel \cite{CookNagel})\label{cook}
 A clique vertex-partition of $G$ is a set 
$\pi=\{W_1,\ldots,W_t\}$ of disjoint (possibly empty) cliques of $G$ such that their disjoint union forms 
$V(G)$. From $G$ and $\pi$, let $G^{\pi}$ denote the graph on the vertex set 
$V(G^{\pi})=V(G) \cup \{w_1,\ldots,w_t\}$ and 
$E(G^{\pi})=E(G) \cup \bigcup\limits_{1 \leq i \leq t}\{\{v,w_i\} \mid v \in W_i\}.$
We call $G^{\pi}$ a clique whiskering of $G$
\end{definition}
Note that $G$ may have many different clique vertex-partitions, and every graph has at least one 
clique vertex-partition namely trivial partition, $\pi=\{\{x_1\},\ldots,\{x_n\}\}$ where 
$V(G)=\{x_1,\ldots,x_n\}$.
 For more details on clique vertex-partition, we refer the reader to \cite{CookNagel}.
 
\begin{theorem}\label{main-clique}
Let $G$ be a graph with clique vertex partition $\pi$.
%Let $\pi=\{W_1,\ldots,W_t\}$ be a clique vertex-partition of $G$. 
 Then
 $J(G^{\pi})^{(k)}$ has linear quotients for all $k \geq 1$.
\end{theorem}
\begin{proof}
Let $\pi=\{W_1,\ldots,W_t\}$ be a clique vertex-partition of $G$ and 
$W_i=\{x_{i_1},\ldots,x_{i_{r_i}}\}$  for all $1 \leq i \leq t$. 
 By proof of  Theorem \ref{main-result2}, it is enough to prove that $G^{\pi}_k$ is vertex decomposable 
 for all $k \geq 1$. We prove this by induction on $m:=k+|V(G)|$. 
 If $k \geq 1$ and $|V(G)|=1$, then by Theorem \ref{star-verdec}, $G_k^{\pi}$ is  vertex decomposable.
 If $k=1$ and $|V(G)| \geq 2$, then by \cite[Theorem 3.3]{CookNagel}, $G_1^{\pi}$ is vertex decomposable.
 
 Assume that $k\geq 2$ and $|V(G)| \geq 2$. For a fixed $1 \leq p \leq t$, let $z_{r_1+\cdots+r_{p-1}+q}=x_{p_q,1}$
 for all $1 \leq q \leq r_p$ and $z_{r_1+\cdots+r_t+l}=w_{l,1}$ for all $1 \leq l \leq t$. Set $\Psi_0=G_k$,
 \[
  \Psi_{i}=\Psi_{i-1} \setminus z_{i} \text{ and } \Omega_i=\Psi_{i-1}\setminus N_{\Psi_{i-1}}[z_i]
  \text{ for all } 1 \leq i \leq r_1+\cdots+r_t+t.
 \]
Since $w_1,\ldots,w_t$ are simplicial vertices of $G^{\pi}$, by Lemmas \ref{simplicial-vertex}, \ref{technical-le},
$z_i$ is a shedding vertex of $\Psi_{i-1}$ for all $1 \leq i \leq r_1+\cdots+r_t$.
Note that $\deg_{\Psi_{r_1+\cdots+r_t}}(x_{i_j,k})=1$ for all $1 \leq i \leq t$, $1 \leq j \leq r_i$.
Since $\{w_{i,1},x_{i_j,k}\} \in E(\Psi_{r_1+\cdots+r_t})$ for all $1 \leq j \leq r_i$, by Corollary \ref{cor-shedding},
$w_{i,1}$ is shedding vertex of $\Psi_{r_1+\cdots+r_t}$. Hence $z_i$ is a shedding vertex of
$\Psi_{i-1}$ for all $r_1+\cdots+r_t \leq i \leq r_1+\ldots+r_t+t$. By Lemma \ref{tech_lemma-cons},
$\Psi_{r_1+\cdots+r_t+t} \simeq G_{k-2}^{\pi} \cup \{\text{ isolated vertices }\}$. By induction on $m$,
$\Psi_{r_1+\cdots+r_t+t}$ is a vertex decomposable graph. 
It follows from Lemma \ref{tech_lemma-cons} that
for any $1 \leq p,l \leq t$, $1 \leq q \leq r_p$, we have
\begin{align*}
 \Omega_{r_1+\cdots+r_{p-1}+q} &\simeq \big(G\setminus N_{G}[x_{p_q}]\big)^{\pi'}_k
 \setminus \big\{z_1,
\ldots ,z_{r_1+\cdots+r_{p-1}+q-1}\big\} \cup \{\text{isolated vertices}\} \text{ and }\\
\Omega_{r_1+\cdots+r_{t}+l} & \simeq \big(G \setminus W_l\big)^{\pi''}_k\setminus \big\{z_1,
\ldots ,z_{r_1+\cdots+r_{t}+l-1}\big\} \cup \{\text{isolated vertices}\},
\end{align*}
where $\pi'=\big\{W_1 \setminus N_{G}[x_{p_q}],\ldots,W_t\setminus N_{G}[x_{p_q}]\big\}$ and 
$\pi''=\big\{W_1,\ldots, W_{l-1},W_{l+1}\ldots,W_t\big\}$.
By induction on $m$, $\big(G\setminus N_{G}[x_{p_q}]\big)^{\pi'}_k$ and 
$\big(G \setminus W_l\big)^{\pi''}_k$ are vertex decomposable graphs. 
Set $\A=\{z_1,\ldots ,z_{r_1+\cdots+r_{p-1}+q-1}\}$. Since every element in $\A$ is either
a neighbor of simplicial vertex or in $(G\setminus N_{G}[x_{p_q}]\big)^{\pi'}_k$, 
by  Corollary \ref{van-noupure}, $\Omega_{r_1+\cdots+r_{p-1}+q}$ is a
vertex decomposable graph. Similarly, we can show that 
$\Omega_{r_1+\cdots+r_{t}+l}$ is a vertex decomposable graph. Hence, by Lemma \ref{obs-verdec}, 
$G^{\pi}_k$ is a vertex decomposable graph for all $k \geq 1$.
\end{proof}

For a monomial ideal $I$, let $\deg(I)$ denote the maximum degree of elements of $G(I)$. 
Thus, in particular, $\deg(J(G))=\max\{|C| : C \text{ is a minimal vertex cover of } G\}.$

\begin{obs}\label{obs-deg} 
Let $G$ be a graph. 
\begin{enumerate}
\item Let $C$ be a minimal vertex cover of $G$. If $x \in C$, then $N_G(x)\nsubseteq C$. 
%Similarly, if $N_G(x) \subseteq C$, then $x \notin C$.
 \item If $\{x_{i_1},\ldots,x_{i_r}\}$ is a minimal vertex cover of 
$G$, then $$\{x_{i_1,1},\ldots,x_{i_1,k},\ldots,x_{i_r,1},\ldots,x_{i_r,k}\}$$ is a minimal
vertex cover of $G_k$ for all $k \geq 1$. Therefore $k \cdot \deg(J(G)) \leq \deg(J(G_k))$.
\item Suppose $G=\kk(x_1)$ is a star complete graph on $\{x_1,y_1,\ldots,y_n\}$. 
Note that $V(G) \setminus x_1$ is the maximum cardinality of minimal vertex cover of $G$ and
$V(G_k) \setminus \{x_{1,1},\ldots,x_{1,k}\}$ is the minimal vertex cover of $G_k$.
By (1), $V(G_k) \setminus \{x_{1,1},\ldots,x_{1,k}\}$ is the maximum cardinality of minimal vertex cover of $G_k$.
Therefore $\deg(J(G_k))=k(|V(G)|-1)= k \cdot \deg(J(G))$ for all $k \geq 1$.
\item Let $\pi=\{W_1,\ldots,W_t\}$ be the clique vertex-partition of $G$. 
We may assume that $W_i \neq \emptyset$ for all $1 \leq i \leq t$.
Note that 
$V(G)$ is the maximum cardinality of minimal vertex cover of $G^{\pi}$. 
Since $W_i$'s are cliques, 
$V(G_k)$ is the maximum cardinality of minimal vertex cover of $G^{\pi}_k$ for all $k \geq 1$.
Therefore $\deg(J(G^{\pi}_k))=k|V(G)|=|V(G_k^{\pi})|-kt=k \cdot \deg(J(G))$ for all $k \geq 1$.
\end{enumerate}

\end{obs}

Since the regularity of a componentwise linear ideal can be computed from its generators, 
we obtain a formula for the regularity of symbolic powers of the vertex cover ideal 
in terms of the maximum size of minimal vertex covers of graph.
\begin{corollary} \label{reg-cor} Let $G$ be a graph. If 
\begin{enumerate}
 \item $G$ is a star complete graph;
 \item $G=H^{\pi}$ for some graph $H$,  or
 \item $G$ is a bipartite graph with satisfies the hypothesis of
Theorem \ref{main-result1},
\end{enumerate}
then for all $k \geq 1$, $
  \reg(J(G)^{(k)})=k \cdot \deg(J(G)). $
 \end{corollary}
\begin{proof}
%It follows from \cite[Lemma 3.1]{Seyed} that $k\cdot\deg(J(G))\leq \reg(J(G)^{(k)})$. We need to prove that
%$\reg(J(G)^{(k)}) \leq k\cdot \deg(J(G))$. 
It follows from  Theorem \ref{main-result1} and Theorem \ref{main-clique} that
$J(G)^{(k)}$ has linear quotients for all $k \geq 1$. Therefore \cite[Corollary 8.2.14]{Herzog'sBook},
$\reg(J(G)^{(k)})$ is equal to the highest degree of a generator in a minimal set of generator
of $J(G)^{(k)}$. 
\vskip 0.1mm \noindent
(1) If $G$ is a star complete graph, then 
by Corollary \ref{pol_reg}, Lemma \ref{fak-result} and Observation \ref{obs-deg}, for all $k \geq 1$,
$
 \reg(J(G)^{(k)}) = \reg(\widetilde{J(G)^{(k)}})=\reg(J(G_k))=\deg(J(G_k))=k \cdot \deg(J(G)).
$
\vskip 0.1mm \noindent
(2) Follows from Corollary \ref{pol_reg}, Lemma \ref{fak-result} and Observation \ref{obs-deg}.
 \vskip 0.1mm \noindent
(3) Since $G$ is a bipartite graph, by \cite[Corollary 2.6]{GRV05},
$J(G)^k=J(G)^{(k)}$ for all $k \geq 1$. The remaining proof is similar to the above case.
\end{proof}

The following example shows that if $G$ is not bipartite, then the assertion of the
Corollary \ref{reg-cor}(3) need not necessarily be true.
\begin{example}\label{ex:notlinearquotients}
Let $I=(x_1x_2,x_2x_3,x_3x_1,x_1x_4,x_1x_5,x_2x_6,x_2x_7,x_3x_8,x_3x_9)$ be an ideal
and $G$ be the associated graph. 
Note that
$ J(G)=(x_1x_2x_3,x_2x_3x_4x_5,x_1x_3x_6x_7,x_1x_2x_8x_9)$ and
 $J(G)^{(2)}=J(G)^2+(x_1x_2x_3x_4x_5x_6x_7x_8x_9).$
Note that $G$ satisfies the hypothesis of Theorem \ref{main-result1}.
It follows from \cite[Theorem 8.2.15 and Corollary 8.2.14]{Herzog'sBook} and 
Theorem \ref{main-result1} that
$\reg(J(G)^{(k)})$ is equal to the highest degree of a generator in a minimal set of generators of
$J(G)^{(k)}$ for all $k \geq 1$. Therefore, $\reg(J(G))=4$ but $\reg(J(G)^{(2)})=9$.

\end{example}

We conclude the paper by raising the following question.
If $G$ is a graph with satisfies the hypothesis  
Theorem \ref{main-result2} or Theorem  \ref{main-result1} or Theorem \ref{main-clique}, then $G_k$ is a 
vertex decomposable graph for all $k \geq 1$. 
As a natural extension of these results, one tend to think that the same expression may hold true 
for vertex decomposable graphs. This is not the case (see, Example \ref{ex:compnt}).
 Therefore, we would like to ask:
 \begin{question}
  For which vertex decomposable graphs $G$, $G_k$ is a vertex decomposable graph for all $k \geq 2$?
  More generally, for which (sequentially) Cohen-Macaulay graphs $G$, $G_k$ is a (sequentially) Cohen-Macaulay graph
  for all $k \geq 2$?
 \end{question}
\vskip 1mm \noindent
\textbf{Acknowledgement:} 
The author would like to thank A. V. Jayanthan for his encouragement
and insightful conversations.
The author extensively used \textsc{Macaulay2} and the packages \textsc{EdgeIdeals}, \cite{FHV_software},
\textsc{SimplicialDecomposability}, \cite{Cook}, \textsc{SymbolicPowers}, \cite{symbolic-package},
for testing his computations. The author would also like to thank the
Institute of Mathematical Sciences, Chennai for financial support.

\bibliographystyle{abbrv}
\bibliography{refs_reg} 
\end{document}